\documentclass[12pt,a4paper]{amsart}
\usepackage{amscd,amssymb,amsopn,amsmath,amsthm,graphics,amsfonts,enumerate,verbatim,calc,textcomp}
\usepackage{pgf,tikz}
\usepackage{mathrsfs}
\usetikzlibrary{arrows}
\usepackage[usenames,dvipsnames]{pstricks}
 \usepackage{epsfig}
 \usepackage{pst-grad} % For gradients
 \usepackage{pst-plot} % For axes
 \usepackage[space]{grffile} % For spaces in paths
\usepackage{etoolbox} % For spaces in paths

\headsep=1cm \numberwithin{equation}{section}
\newcommand{\To}{\rightarrow}

\newcommand{\D}{\mathcal{D} }

\newcommand{\Z}{\mathbb{Z} }

\newcommand{\CA}{\mathcal{A} }

\newcommand{\CF}{\mathcal{F} }

\newcommand{\CT}{\mathcal{T} }
\newcommand{\CX}{\mathcal{X} }

\newcommand{\CU}{\mathcal{U} }
\newcommand{\CV}{\mathcal{V} }

\newcommand{\im}{{\rm{Im}}}

\newcommand{\Ker}{{\rm{Ker}}}

\newcommand{\Spec}{{\rm{Spec}}}
\newcommand{\ASpec}{{\rm{ASpec}}}

\newcommand{\MinASupp}{{\rm{MinASupp}}}

\newcommand{\ASupp}{{\rm ASupp}}
\newcommand{\AAss}{{\rm AAss}}

\newcommand{\Hom}{{\rm{Hom}}}
\newcommand{\Ext}{{\rm{Ext}}}

\newtheorem{theorem}{Theorem}[section]
\newtheorem{corollary}[theorem]{Corollary}
\newtheorem{lemma}[theorem]{Lemma}
\newtheorem{proposition}[theorem]{Proposition}

\theoremstyle{definition}
\newtheorem{definition}[theorem]{Definition}

\newtheorem{definitions and notations}[theorem]{Definitions and Notations}

\theoremstyle{plain}

\theoremstyle{definition}

\numberwithin{equation}{section}

\begin{document}

\title{Local cohomology in Grothendieck categories}
\author{Fatemeh Savoji and Reza Sazeedeh}
\address{Department of Mathematics, Urmia University, P.O.Box: 165, Urmia, Iran}
\email{F.savoji@urmia.ac.ir}

\address{Department of Mathematics, Urmia University, P.O.Box: 165, Urmia, Iran}
\email{rsazeedeh@ipm.ir}

\makeatletter
\@namedef{subjclassname@2010}{%
  \textup{2010} Mathematics Subject Classification}
\makeatother



\begin{abstract}
Let $\CA$ be a locally noetherian Grothendieck category. In this paper we define and study the section functor on $\CA$
with respect to an open subset of $\ASpec\CA$. Next we define and study local cohomology theory in $\CA$ in terms of the section functors. Finally we study abstract local cohomology functor on the derived category $\mathcal{D}^+(\CA)$.   
\end{abstract}
\maketitle

\section{introduction}
Throughout this paper $\CA$ is a locally noetherian Grothendieck category. The main aim of this paper is to define and study local cohomology notion in Grothendieck categories. We define local cohomology with respect to an open subset of atom spectrum of $\CA$, $\ASpec\CA$, defined by Kanda [K1, K2, K3]. To be more precise, for any open subset $W$ of $\ASpec\CA$, we define the section functor $\Gamma_W$ on $\CA$ and we show that they are in corresponding to the left exact radical functors, a classical notion in Grothendieck categories.

Section 2 is devoted to some backgrounds about monoforms objects, atoms and atom spectrum. We obtain a result about atom support of monoform objects. We show if $\ASpec\CA$ is Alexandroff, then every monoform object $H$ with $\alpha=\overline{H}$ contains a monoform subobject $H_1$ such that $\ASupp(H_1)=\{\beta\in\ASpec\CA|\hspace{0.1cm} \alpha\leq \beta\}$. 

In section 3, we define and study preradical and radical functors on $\CA$. We find a characterization of left exact radical functors. We prove that if $\gamma$ is a left exact preradical functor, then it  preserves injective objects if and only if $\CT_{\gamma}$ is stable where $\CT_{\gamma}$ is the pretorsion class induced by $\gamma$ (cf. Proposition 3.3). 
In Theorem \ref{adjinj}, we prove that if $\gamma$ is a left exact radical functor with the corresponding torsion theory $(\CT_{\gamma},\CF_{\gamma})$ and if $M$ is a $\gamma$-torsion-free object of $\CA$, then up to isomorphisms, $GF(M)$ is the smallest faithfully $\gamma$-injective object containing $M$ where $F:\CA\To \CA/\CT_{\gamma}$ is the canonical functor with its right adjoint functor $G:\CA/\CT_{\gamma}\To \CA$. This theorem immediately concludes that $H(\alpha)$ is a faithfully $t_{\alpha}$-injective object for any $\alpha\in\ASpec\CA$ where $t_{\alpha}$ is the left exact radical functor corresponding to $\CX(\alpha)=\ASupp^{-1}(\ASpec\CA\setminus\overline{\{\alpha\}})$. 

Given an object $M$ of $\CA$ and an open subset $U$ of $\ASpec\CA$, the  torsion subobject of $M$ with respect to $U$ is denoted by 
$\Gamma_U(M)$ which is the largest subobject of $M$ such that $\ASupp(\Gamma_U(M))\subset U$. We prove in Theorem \ref{radgam} that any left exact preradical functor $\gamma$ is a subfunctor of $\Gamma_{U_{\gamma}}$, where $$U_{\gamma}=\{\alpha\in\ASpec\CA|\hspace{0.1cm} H\in\CT_{\gamma}\hspace{0.1cm} {\rm for}\hspace{0.1cm} {\rm some}\hspace{0.1cm} {\rm monoform} \hspace{0.1cm}{\rm object}\hspace{0.1cm} H\in\alpha\}.$$ Moreover, if $\gamma$ is radical, then the equality $\gamma=\Gamma_{U_{\gamma}}$ holds. We also prove that if $\Gamma_U$ preserves injective objects and $M$ is an object of $\CA$, then $\AAss(M/\Gamma_U(M))\subseteq\AAss(M)$ (cf. Proposition \ref{asssec}). 

In Section 4, we define the dimension of an object $M$ of $\CA$  in terms of atoms in $\ASupp(M)$. We obtain a relationship between this new dimension and the classical dimension given by Gabriel [Ga]. We also study  local cohomology of objects of $\CA$ with respect to an open subset of $\ASpec\CA$. We obtain a result about non-vanishing of local cohomology of objects of $\CA$. To be more precise, in Theorem \ref{nonvanish}, we show that if $(\CA,\mathfrak m)$ is a local category, $N$ is a noetherian object of $A$ of dimension one and for any $\alpha\in\AAss(N/\Gamma_{\mathfrak m}(N))$ there exists a monoform object $M$ such that ${\rm End}(M)$ is not a skew field and $\alpha=\overline{M}$, then $H_{\frak m}^1(N)$ is not noetherian.

In Section 5, we study abstract local cohomology on the category $\CA$. Our idea goes back to a work by Yoshino and Yoshizawa [YY]  on the category  of $R$-Mod, when $R$ is a commutative ring. In Proposition \ref{radab}, we show that if $\gamma$ is a left exact radical functor preserving injective objects, then ${\bf R}\gamma$ is an abstract local cohomology. Finally we prove that if $\delta$ is an abstract local cohomology on $\D^+(\CA)$, then there exists an open subset $W$ of $\ASpec\CA$ such that $\delta={\bf R}\Gamma_W$
(cf. Theorem \ref{tttt}).

%%%%%%%%%%%%%%%%%%%%%%%%%%%%%%%%%%%%%%%%%%%%%%%%%%%%%%%%%%%%%%%%%%%%%%%%%%%%%%%%%%%%%%%%%%%%%%%%%%%%%%%%%%%
%%%%%%%%%%%%%%%%%%%%%%%%%%%%%%%%%%%%%%%%%%%%%%%%%%%%%%%%%%%%%%%%%%%%%%%%%%%%%%%%%%%%%%%%%%%%%%%%%%%%%%%%%%%
%%%%%%%%%%%%%%%%%%%%%%%%%%%%%%%%%%%%%%%%%%%%%%%%%%%%%%%%%%%%%%%%%%%%%%%%%%%%%%%%%%%%%%%%%%%%%%%%%%%%%%%%%%%
\section{Monoform objects and their atom support}

We first recall from [K1] the definition of monoform objects and
atoms spectrum in a Grothendieck category. An abelian category $\CA$
is called {\it Grothendieck} if it has exact direct limits and a
generator.

\begin{definition}
(i) A nonzero object $M$ in $\CA$ is {\it monoform} if for any
nonzero subobject $N$ of $M$, there exists no common nonzero
subobject of $M$ and $M/N$ which means that there does not exist a
nonzero subobject of $M$ which is isomorphic to a subobject of
$M/N$. We denote by $\ASpec _0\CA$, the set of all monoform
objects of $\CA$.

 (ii) Two monoform objects $H$ and $H'$ are said to be {\it
atom-equivalent} if they have a common nonzero subobject.

(iii) By [K1, Proposition 2.8], the atom equivalence establishes an
equivalence relation on monoform objects; and hence for every
monoform object $H$, we denote the
 {\it equivalence class} of $H$, by $\overline{H}$, that is
\begin{center}
  $\overline{H}=\{G\in\ASpec _0\CA|\hspace{0.1cm} H \hspace{0.1cm}
 {\rm and}\hspace{0.1cm} G$ has a common nonzero subobject$\}.$
\end{center}

 (iv) The {\it atom spectrum} $\ASpec\CA$ of $\CA$ is the quotient set
of $\ASpec _0\CA$ consisting of all equivalence classes induced by
this equivalence relation. Any equivalence class is called an {\it
atom} of $\ASpec\CA$.

(v) For an object $M$ of $\CA$, we define a subset $\ASupp(M)$ of
$\ASpec\CA$ by
$$\ASupp M=\{\alpha\in\ASpec\CA|\hspace{0.1cm} {\rm there\hspace{0.1cm} exists\hspace{0.1cm}}
 H\in\alpha\hspace{0.1cm} {\rm which \hspace{0.1cm}is\hspace{0.1cm} a
 \hspace{0.1cm}subquotient\hspace{0.1cm} of\hspace{0.1cm}} M\}.$$

We also define {\it the associated atoms} of $M$, denoted by
$\AAss(M)$, a subset of $\ASupp(M)$ that is
$$\AAss M=\{\alpha\in\ASupp(M)|\hspace{0.1cm} {\rm there\hspace{0.1cm} exists\hspace{0.1cm}}
 H\in\alpha\hspace{0.1cm} {\rm which \hspace{0.1cm}is\hspace{0.1cm} a
 \hspace{0.1cm}subobject\hspace{0.1cm} of\hspace{0.1cm}} M\}.$$

 (v) A subset $\Phi$ of $\ASpec\CA$ is called {\it open} if for any
$\alpha\in\Phi$, there exists $H\in\alpha$ such that
$\ASupp(H)\subset\Phi$. For any nonzero object $M$ of $\CA$, it is
clear that $\ASupp(M)$ is an open subset of $\CA$.
 \end{definition}

 \medskip
We recall the definition of Serre subcategories and the quotient category induced by a Serre subcategory. 
\begin{definition}
A full subcategory $\CX$ of an abelian category $\CA$ is called {\it Serre} if for any exact sequence $0\To M\To N\To K\To 0$ of $\CA$, the object $N$ belongs to $\CX$ if and only if $M$ and $K$ belong to $\CX$. The quotient category $\CA/\CX$ of $\CA$ induced by $\CX$ is defined as follows.
\begin{enumerate}
	\item The objects of $\CA/\CX$ and $\CA$ are the same. 
	\item For any objects $M$ and $N$ in $\CA$ we have
	$$\Hom_{\CA/\CX}(M,N)= \underset{\underset{(M',N')\in\mathcal{S}(M,N)}{\longrightarrow}}{\mbox{lim}}\Hom_{\CA}(M',N/N')$$
	where $\mathcal{S}(M,N)$ is the direct set defined by $$\mathcal{S}(M,N)=\{(M',N')|\hspace{0.1cm}M'\subset M, N'\subset N \hspace{0.1cm}{\rm such\hspace{0.1cm}  that} \hspace{0.1cm}M/M',N'\in\CX\}$$ and for $(M',N'), (M'',N'')\in\mathcal{S}(M,N)$, we have $(M',N')\leq (M'',N'')$ if $M''\subset M'$ and $N'\subset N''$. 
		\item Let $L, M, N$ be objects in $\CA$ and $[f]\in\Hom_{\CA/\CX}(L,M)$ and $[g]\in\Hom_{A/\CX}(M,N)$. Assume that $[f]$ and $[g]$ are represented by $f\in\Hom_{\CA}(L',M/M')$ and $g\in\Hom_{\CA}(M'',N/N'')$ where $(L',M')\in\mathcal{S}(L,M)$ and 
		$(M'',N'')\in\mathcal{S}(M,N)$, respectively. Then the composite $[g]\circ[f]\in\Hom_{\mathcal{S}}(L,N)$ is the equivalence class of the composite of $f'':f^{-1}(\frac{M'+M''}{M'})\To \frac{M'+M''}{M'}$ and $g':\frac{M'+M''}{M'}\To N/N'$ where $f''$ and $g'$ are the induced morphisms by $f$ and $g$, respectively and $N'/N''=g(M'\cap M'')$.
	\end{enumerate}
	In this case, we can define the canonical additive functor $F:\CA\To\CA/\CX$ by the assignment $M\mapsto M$  for each object $M$ of $\CA$ and the canonical map $\Hom_{\CA}(M,N)\To\Hom_{\CA/\CX}(M,N)$ for objects $M$ and $N$ in $\CA$.
	
		The Serre subcategory $\CX$ of $\CA$ is called {\it localizing} if the canonical functor $F$ has a right adjoint functor. 
\end{definition}
 
\medskip

We recall from [K2] that $\ASpec\CA$ can be regarded as a partially
ordered set together with a specialization order $\leq$ as follows:
for any atoms $\alpha$ and $\beta$ in $\ASpec\CA$, we have
$\alpha\leq \beta$ if and only if for any open subclass $\Phi$ of
$\ASpec\CA$ satisfying $\alpha\in \Phi$, we have $\beta\in\Phi$.

\medskip

For every $\alpha\in\ASpec\CA$, the topological closure of $\alpha$, denoted by $\overline{\{\alpha\}}$ consists of all
 $\beta\in\ASpec\CA$ such that $\beta\leq\alpha$. According to [K1, Theorem 5.7], for each atom $\alpha$, there exists a localizing subcategory
 $\CX(\alpha)=\ASupp^{-1}(\ASpec\CA\setminus\overline{\{\alpha\}})$ induced by $\alpha$, where $\ASupp^{-1}(U)=\{M\in\CA|\hspace{0.1cm}\ASupp(M)\subseteq U\}$ for any subset $U$ of $\ASpec\CA$. We denote by $\CA_{\alpha}$ the quotient category $\CA/\CX(\alpha)$. 

For any object $M$ of $\CA$, we denote $F(M)$ by $M_{\alpha}$ where $F:\CA\To\CA_{\alpha}$ is the canonical functor. We also remember from [K1, Proposition 5.10] that $$\ASupp(M)=\{\alpha\in\ASpec\CA|\hspace{0.1cm}M_{\alpha}\neq 0\}$$ and using [K1,Proposition 5.5], we have $$\ASupp(M_{\alpha})=\{\beta \in\ASupp(M)
|\hspace{0.1cm}\beta\leq \alpha\}.$$

A topological space $X$ is called {\it Alexandroff} if the intersection of any family of open
subsets of $X$ is also open.

\medskip
\begin{proposition}\label{alex}
If $\alpha\in\ASpec\CA$, then $\bigcap_{H\in\alpha}\limits\ASupp(H)=\{\beta\in\ASpec\CA|\hspace{0.1cm} \alpha\leq \beta\}$. In particular if $\ASpec\CA$ is Alexandroff, then there exists a monoform object $H$ with $\alpha=\overline{H}$ and $\ASupp(H)=\{\beta\in\ASpec\CA|\hspace{0.1cm} \alpha\leq \beta\}$.
\end{proposition}
\begin{proof}
The first claim has been proved by Kanda [K2, Proposition 4.2]. In order to prove the second claim, since $\CA$ is Alexandroff, $\{\beta\in\ASpec\CA|\hspace{0.1cm} \alpha\leq \beta\}$ is open and so there exists an object $M$ of $\CA$ such that $\{\beta\in\ASpec\CA|\hspace{0.1cm} \alpha\leq \beta\}=\ASupp(M)$. Since $\alpha\in\ASupp(M)$, the object $M$ has a subquotient $M/K$ containing a monoform subobject $H$ with $\overline{H}=\alpha$. It is clear by the first part that $\ASupp(H)=\{\beta\in\ASpec\CA|\hspace{0.1cm} \alpha\leq \beta\}$. 
\end{proof}

\medskip
\begin{corollary}
Let $\ASpec\CA$ be an Alexandroff topological space and $\alpha\in\ASpec\CA$. Then every monoform object $H$ with $\alpha=\overline{H}$ contains a monoform subobject $H_1$ such that $\ASupp(H_1)=\{\beta\in\ASpec\CA|\hspace{0.1cm} \alpha\leq \beta\}$.
 \end{corollary}
\begin{proof}
According to Proposition \ref{alex}, there exists a monoform object  $H'$ such that $\alpha=\overline{H'}$ and $\ASupp(H')=
 \{\beta\in\ASpec\CA|\hspace{0.1cm} \alpha\leq \beta\}$. It is clear by Proposition \ref{alex} that $\ASupp(H')=\ASupp(H'')$ for any 
subobject $H''$ of $H'$. Furthermore, any monoform object $H$ with $\alpha=\overline{H}$ and $H'$ have a common nonzero subobject $H_1$ which satisfies our claim.
\end{proof}

%%%%%%%%%%%%%%%%%%%%%%%%%%%%%%%%%%%%%%%%%%%%%%%%%%%%%%%%%%%%%%%%%%%%%%%%%%%%%%%%%%%%%%%%%%%%%%%%%%%%%%%%%%%%%%%%%%%
%%%%%%%%%%%%%%%%%%%%%%%%%%%%%%%%%%%%%%%%%%%%%%%%%%%%%%%%%%%%%%%%%%%%%%%%%%%%%%%%%%%%%%%%%%%%%%%%%%%%%%%%%%%%%%%%%%%
%%%%%%%%%%%%%%%%%%%%%%%%%%%%%%%%%%%%%%%%%%%%%%%%%%%%%%%%%%%%%%%%%%%%%%%%%%%%%%%%%%%%%%%%%%%%%%%%%%%%%%%%%%%%%%%%%%%

\section{Preradical functors in Grothendieck categories}

Let ${\bf 1}:\CA\To\CA$ be the identity functor. Then ${\bf 1}$ is an object of the functor category ${\rm Fun}(\CA,\CA)$. Any subobject $\gamma$ of ${\bf 1}$ is called {\it preradical}. In other words, $\gamma(M)$ is a subobject of $M$ for any object $M$ of $\CA$ and for any morphism $f:M\To N$ of objects of $\CA$, the morphism $\gamma(f)$ is a restriction of $f$ onto $\gamma(M)$. The preradical $\gamma$ is called radical, if for any object $M$ of $\CA$, we have $\gamma(M/\gamma(M))=0$ and $\gamma$ is idempotent if $\gamma^2=\gamma$ (i.e. $\gamma(\gamma(M))=\gamma(M)$ for any object $M$).

For any preradical functor $\gamma$ on $\CA$, we define {\it pretorsion (or $\gamma$-pretorsion) class} $\CT_{\gamma}$ and {\it pretorsion-free (or $\gamma$-pretorsion-free) class} $\CF_{\gamma}$ as follows

$$\CT_{\gamma}=\{T\in\CA| \hspace{0.1cm}\gamma(T)=T\}$$
 $$\CF_{\gamma}=\{F\in\CA|\hspace{0.1cm} \gamma(F)=0\}.$$
We notice that each of element of $\CT_{\gamma}$ is called $\gamma$-{\it torsion} and each element of $\CF_{\gamma}$ is called $\gamma$-torsion-free. 
 
\medskip
Assume that $F:\CA\To \mathcal{D}$ and $F':\CA\To \mathcal{D}'$ are exact functors admitting full and faithful right adjoint functors. We say that $F\sim F'$ if there exists a unique equivalence functor $H:\mathcal{D}\To\mathcal{D}'$ such that $H\circ F\simeq F'$ ($\simeq$ shows the natural equivalence). It is clear  that $\sim$ is an equivalence relation and if $F\sim F'$, then $\Ker(F)=\Ker(F')$. Denoting the equivalence class of $F$ by $[F]$ and using [P, Chap 4, Theorem 4.9], for any exact functor $F:\CA\To \mathcal{D}$ admitting a full and faithful right adjoint functor, we have $[F]=[F_{\Ker(F)}]$ where $F_{\Ker(F)}:\CA\To\CA/\Ker(F)$ is the canonical exact functor. We put 
$$\mathscr{B}=\{[F]| F:\CA\To\mathcal{D} {\rm \hspace{0.1cm}is \hspace{0.1cm}an \hspace{0.1cm}exact\hspace{0.1cm} functor\hspace{0.1cm} }$$$${\rm admitting \hspace{0.1cm} a\hspace{0.1cm} full\hspace{0.1cm}
 and\hspace{0.1cm} faithful\hspace{0.1cm} right\hspace{0.1cm} adjoint\hspace{0.1cm} functor}\}.$$ 
%/When $\CA$ has enough injective objects, the following theorem specifies the left exact radicals.

\begin{proposition}
 Let $\CA$ have enough injective objects. Then there is a one-to-one correspondence between the class of left exact radical functors of $\CA$ and $\mathscr{B}$.
\end{proposition}
\begin{proof}
Assume that $r$ is a left exact radical of $\CA$. Then it is clear that for any object $X$ of $\CA$, $r(X)$ is the largest subobject of $X$ belonging to $\CT_r$ and hence $\CT_r$ is a localizing subcategory of $\CA$ by [P, Chap 4, Proposition 5.2]. On the other hand, if $F:\CA\To\CA'$ is an exact functor admitting a full and faithful right adjoint functor, then according to [P, Chap 4, Theorem 4.9], $\Ker(F)$ is a localizing subcategory. Assume that $r_{\Ker(F)}(X)$ is the largest subobject of $X$ belonging to $\Ker(F)$. It follows from [St, Chap VI, Proposition 1.7] that $r_{\Ker(F)}$ is a left exact radical functor.  
Consider $\mathscr{A}=\{r:\CA\To \CA|\hspace{0.1cm} r \hspace{0.1cm}{\rm is\hspace{0.1cm} a \hspace{0.1cm}left\hspace{0.1cm} exact\hspace{0.1cm} radical\hspace{0.1cm} functor}\}$. We define a map $\Phi:\mathscr{A}\To \mathscr{B}$ by $\Phi(r)=[F_r]$, where $F_r:\CA\To\CA/\CT_r$ is the canonical exact functor. We also define $\Theta:\mathscr{B}\To\mathscr{A}$ 
by $\Theta([F])=r_{\Ker(F)}$. One can show that $\Theta\circ\Phi=1_{\mathscr{A}}$ and $\Phi\circ\Theta=1_{\mathscr{B}}$ as $\CT_{r_{\Ker(F)}}=\Ker(F)$ and $r_{\Ker(F_r)}=r$.
\end{proof}

\medskip

Let $E$ be an injective object of $\CA$. Following [V], for any object $C$ of $\CA$, we define 
$$t_E(C)=\bigcap \{\ker f| f\in\Hom_{\CA}(C,E)\}.$$
As according to [V, Proposition 3.2], $t_E$ is a left exact radical functor, $\CT_{t_E}$ is localizing and it is called the localizing subcategory of $\CA$ {\it generated by} $t_E$.

 Let $\alpha\in\ASpec\CA$ and $E=E(\alpha)$.  Ahmadi et al. [AS, Theorem 2.11] showed that the left exact radical functors $t_E$ and $t_{\alpha}$ generate the same localizing subcategories where $t_{\alpha}$ is the left exact radical functor corresponding to $\CX(\alpha)$. To be more precise, $\CT_{t_{E}}=\ASupp^{-1}(\ASpec\CA\setminus\overline{\{\alpha\}})$. In the following proposition we show that $t_E=t_{\alpha}$.

\medskip

\begin{proposition}
Let $E$ be an indecomposable injective object of $\CA$ with $\AAss(E)=\{\alpha\}$. Then $t_E=t_{\alpha}$.
\end{proposition}  
\begin{proof}
For any object $C$, we have $t_E(C)\in \CT_{t_{E}}$ and so by the above argument $t_{\alpha}(t_E(C))=t_E(C)$. Now since $t_{\alpha}$ is left exact, $t_{\alpha}(t_E(C))=t_E(C)\cap t_{\alpha}(C)$ and hence $t_E(C)\subseteq t_{\alpha}(C)$. Symmetrically we have $t_{\alpha}(C)\subseteq t_E(C)$.
\end{proof}

\medskip
 The following lemma is frequently used in this paper. 
\begin{lemma}\label{simp}{\rm([K4, Proposition 3.5])}
Let $\alpha\in\ASpec\CA$. Then $H_{\alpha}$ is simple in $\CA_{\alpha}$ for any monoform object $H$ with $\alpha=\overline{H}$. 
\end{lemma}
\medskip
From [P], a pretorsion class $\CX$ of
the Grothendieck category $\CA$ is called {\it stable} if the
injective envelope in $\CA$ of any object of $\CX$ is also an object
of $\CX$. Further a Grothendieck category is said to be {\it locally
stable} if any its localizing subcategory is stable. We now have the following result.
\medskip
\begin{proposition}\label{stab}
Let $\gamma$ be a left exact preradical functor. Then $\gamma$ preserves injective objects if and only if $\CT_{\gamma}$ is stable.  
\end{proposition}
\begin{proof}
In order to prove ''only if'', since $\CA$ is locally noetherian, for any object $M$ of $\CA$, we have 
$E(M)=\bigoplus_{\alpha\in\AAss(M)}\limits E(\alpha)^{(\mu(\alpha))}$, where $E(\alpha)^{(\mu(\alpha))}=\bigoplus_{\mu(\alpha)}\limits E(\alpha)$ and $\mu(\alpha)$ denotes the numbers of $E(\alpha)$ in $E(M)$. Hence it suffices to show that for any monoform object $H\in\CT_{\gamma}$ we have $E(H)\in\CT_{\gamma}$. By the assumption, $\gamma(E(H))$ is injective and so it is a direct summand of $E(H)$. But since $E(H)$ is indecomposable, we have $\gamma(E(H))=E(H)$. To prove the converse, first assume that $E$ is an indecomposable  injective object of $\CA$. If $\gamma(E)=0$, there in nothing to prove. If $\gamma(E)$ is nonzero, it is an essential subobject of $E$, hence $E(\gamma(E))=E\in\CT_{\gamma}$ as $\gamma(E)\in\CT_{\gamma}$. Thus $\gamma(E)=E$. We now assume that $E$ is any injective object of $\CA$. By Matlis structure theorem $E=\bigoplus_{i\in\Lambda}\limits E(\alpha_i)$ where $\alpha_i\in\ASpec\CA$ for each $i$. By the first argument we divide $\Lambda$ to two sets $\Lambda_1=\{i\in\Lambda|\hspace{0.1cm}\gamma(E(\alpha_i))=E(\alpha_i)\}$ and $\Lambda_2=\{i\in\Lambda|\hspace{0.1cm}\gamma(E(\alpha_i))=0\}$ 
and so if we set $E_1=\bigoplus_{i\in\Lambda_1}\limits E(\alpha_i)$ and $E_2=\bigoplus_{i\in\Lambda_2}\limits E(\alpha_i)$, then $E=E_1\oplus E_2$. We observe that since $\CT_{\gamma}$ is closed under arbitrary direct sums, $E_1\in\CT_{\gamma}$ and since $\CF_{\gamma}$ is closed under subobjects and products, $E_2\in\CF_{\gamma}$ and hence $\gamma(E)=E_1$.
\end{proof}
 The following lemma show that if a left exact preradical functor preserves injective objects, the it divides indecomposable injective objects.   
\medskip
\begin{lemma}\label{monpres}
Let $\gamma$ be a left exact preradical functor preserving injective objects and let $\alpha\in\ASpec\CA$. Then the following conditions hold.
\begin{itemize}
		\item [(1)] $\gamma(E(\alpha))$ is either $E(\alpha)$ or zero.
	\item [(2)] $\alpha\subseteq \CT_{\gamma}$ or $\alpha\subseteq \CF_{\gamma}$.
 \end{itemize}
  \end{lemma}
\begin{proof}
(1) Since $\gamma(E(\alpha))$ is an injective subobject of $E(\alpha)$, it is a direct summand of $E(\alpha)$ and since $E(\alpha)$ is indecomposable, $\gamma(E(\alpha))=E(\alpha)$ or $\gamma(E(\alpha))=0$. (2) If $H$ is a monoform object in $\CT_{\gamma}$ with $\overline{H}=\alpha$, then $H\subseteq \gamma(E(\alpha))$ and so using the first part $\gamma(E(\alpha))=E(\alpha)$ so that $E(\alpha)\in\CT_{\gamma}$. This implies that $\alpha\subseteq \CT_{\gamma}$. If $H$ is a monoform object in $\CF_{\gamma}$, the equalities $0=\gamma(H)=H\cap\gamma(E)$ and (1) forces that $\gamma(E)=0$. Now, for any monoform object $H'$, since $H'\subseteq E$, we have $\gamma(H')=0$. 
\end{proof}
\medskip
\begin{definition}
Let $\gamma$ be a left exact radical functor. An object $D$ of $\CA$ is called {\it $\gamma$-injective} if $\Ext_{\CA}^1(X,D)=0$ for any $\gamma$-torsion object $X$ of $\CA$. The object $D$ is called {\it faithfully $\gamma$-injective} if $D$ is $\gamma$-injective and $\gamma$-torsion-free. 
\end{definition}
\medskip
We show that any $\gamma$-torsion-free object is embedded in a faithfully $\gamma$-injective object.

\begin{proposition}\label{inj}
Let $\gamma$ be a left exact radical functor and let $M$ be a $\gamma$-torsion-free object of $\CA$. Then there is a faithfully $\gamma$-injective object $D$ containing $M$ such that $D/M$ is $\gamma$-torsion.
\end{proposition}
\begin{proof}
Suppose that $E$ is the injective envelope of $M$ and $X=E/M$ and also $K=X/\gamma(X)$. Then the exact sequences $0\To M\To E\To X\To 0$ and $0\To\gamma(X)\To X\To K\To 0$ of objects induce the following pull back diagram

\begin{center}\setlength{\unitlength}{4cm}
\begin{picture}(4.5,1.5)

\put(0.3,0.5){\makebox(0,0){$0$}}
 \put(0.8,0.5){\makebox(0,0){$M$}}
\put(1.3,0.5){\makebox(0,0){$D$}}
 \put(1.8,0.5){\makebox(0,0){$\gamma(X)$}}
\put(2.3,0.5){\makebox(0,0){$0$}}

\put(0.4,0.5){\vector(1,0){0.3}}
\put(0.9,0.5){\vector(1,0){0.3}}
 \put(1.4,0.5){\vector(1,0){0.25}}
\put(1.95,0.5){\vector(1,0){.3}}

\put(0.3,0.9){\makebox(0,0){$0$}}
 \put(0.8,0.9){\makebox(0,0){$M$}}
\put(1.3,0.9){\makebox(0,0){$E$}}
 \put(1.8,0.9){\makebox(0,0){$X$}}
\put(2.3,0.9){\makebox(0,0){$0$}}

\put(0.4,0.9){\vector(1,0){0.3}}
\put(0.9,0.9){\vector(1,0){0.3}}
 \put(1.4,0.9){\vector(1,0){0.3}}
\put(1.95,0.9){\vector(1,0){.3}}

 \put(0.8,0.5850891){\line(0,1){0.25}}
 \put(0.78,0.5850891){\line(0,1){0.25}}
 \put(1.3,0.5850891){\vector(0,1){0.25}}
\put(1.8,0.585091){\vector(0,1){0.25}}

\put(1.3,0.2){\makebox(0,0){$0$}}
\put(1.8,0.2){\makebox(0,0){$0$}}

\put(1.3,0.25){\vector(0,1){0.2}}
\put(1.8,0.25){\vector(0,1){0.2}}

\put(1.3,0.9535){\vector(0,1){0.235}}
\put(1.8,0.9535){\vector(0,1){0.235}}

\put(1.3,1.25){\makebox(0,0){$K$}}
\put(1.8,1.25){\makebox(0,0){$K$}}

\put(1.35,1.25){\line(1,0){0.3654}}
\put(1.35,1.232){\line(1,0){0.3654}}

\put(1.3,1.32){\vector(0,1){0.2}}
\put(1.8,1.32){\vector(0,1){0.2}}

\put(1.3,1.576){\makebox(0,0){$0$}}
\put(1.8,1.576){\makebox(0,0){$0$}}
\end{picture}
\end{center}

We observe that $E$ is $\gamma$-torsion-free  and then so is $D$. On the other hand, for any $\gamma$-torsion object $N$, applying the functor $\Hom_{\CA}(N,-)$ to the exact sequence $0\To D\To E\To K\To 0$, we deduce that $\Ext_{\CA}^1(N,D)=0$.
\end{proof}

\medskip

\begin{theorem}\label{adjinj}
Let $\gamma$ be a left exact radical functor with the corresponding torsion theory $(\CT_{\gamma},\CF_{\gamma})$. If $M$ is a $\gamma$-torsion-free object of $\CA$, then, up to isomorphisms, $GF(M)$ is the smallest faithfully $\gamma$-injective object containing $M$ where $F:\CA\To \CA/\CT_{\gamma}$ is the canonical functor with its right adjoint functor $G:\CA/\CT_{\gamma}\To \CA$.  
\end{theorem}
\begin{proof}
According to Proposition \ref{inj} there exists a faithfully $\gamma$-injective object $D$ containing $M$ such that $D/M$ is $\gamma$-torsion. By the proof of the above proposition, $D$ is contained in $E(M)$ and so it follows from [K2, Theorem 5.11] that $D\subseteq GF(M)$. On the other hand, since $GF(M)/D$ is a quotient of $GF(M)/M$, it is $\gamma$-torsion. Thus the exact sequence $0\To D\To GF(M)\To GF(M)/D\To 0$ splits so that $D=GF(M)$. In order to prove that $GF(M)$ is the smallest faithfully $\gamma$-injective object containing $M$, assume that $D_1$ is any faithfully $\gamma$-injective object containing $M$ with the inclusion morphism $f:M\To D_1$. By the first part, there exists an exact sequence of objects 
$0\To M\stackrel{h}\To GF(M)\To X\To 0$ such that $X=GF(M)/M$ is $\gamma$-torsion. Then applying the functor $\Hom_{\CA}(-,D_1)$ to the short exact sequence, there exists a morphism $\theta:GF(M)\To D_1$ such that $\theta\circ h=f$ and since $M$ is an essential subobject of $GF(M)$, it is clear that $\theta$ is injective.  
\end{proof}

\medskip
For any atom $\alpha\in\ASpec\CA$, according to [K3, Theorem 3.6], the indecomposable injective object $E(\alpha)$ contains a unique maximal monoform subobject $H(\alpha)$. We now have the following corollary about $H(\alpha)$.

\begin{corollary}
For any $\alpha\in\ASpec\CA$, $H(\alpha)$ is a faithfully $t_{\alpha}$-injective object.
\end{corollary}
\begin{proof}
It is clear that $H(\alpha)$ is $t_{\alpha}$-torsion free and so $H(\alpha)=GF(H(\alpha))$. Therefore the assertion follows by the previous theorem. 
\end{proof}

\medskip

\begin{definition}
Let $M$ be an object of $\CA$ and $U$ be an open subset of $\ASpec\CA$. The {\it torsion subobject of $M$ with respect to $U$} , denoted by 
$\theta_M:\Gamma_U(M)\To M$, is the largest subobject of $M$ such that $\ASupp(\Gamma_U(M))\subset U$. We notice that $\Gamma_U$ is called a {\it section functor on $\CA$ with respect to} $U$. According to the definition, it is clear that $\Gamma_U$ is an idempotent functor (i.e. $\Gamma_U^2=\Gamma_U$).
\end{definition}

For a Grothendieck category, it is well known that there is a bijection between the left exact preradical functors and the hereditary 
pretorsion classes, and the left exact radical functors correspond to the hereditary torsion classes (see, for example, [St, Chap. VI, Corollary 1.8]). The bijection between the hereditary torsion classes (also called localizing subcategories) and open subsets of the atom spectrum is explicitly described in [K2, Theorem 5.10]. The section functor $\Gamma_U$ with respect to an open subset $U$ of $\CA$ can be obtained by combining these two bijections and so we deduce that $\Gamma_U$ is a left exact radical functor.

\medskip
\begin{lemma}\label{gammon}
Let $U$ be an open subset of $\ASpec\CA$ such that $\Gamma_U$ preserves injective objects. Then $\Gamma_U(E(\alpha))=E(\alpha)$ if and only if $\alpha\in U$.
\end{lemma}
\begin{proof}
If $\alpha\in U$, there exists a monoform object $H$ such that $\ASupp(H)\subseteq U$. Therefore $\Gamma_U(H)=H$ so that $H\subseteq\Gamma_U(E(\alpha))$. Therefore it follows from Lemma \ref{monpres} that $\Gamma_U(E(\alpha))=E(\alpha)$. The converse is clear as $\alpha\in\ASupp(E(\alpha))$.
\end{proof}

\medskip
\begin{definition}\label{dd}
For every left exact preradical functor $\gamma$ on $\CA$, we define $$U_{\gamma}=\{\alpha\in\ASpec\CA|\hspace{0.1cm} H\in\CT_{\gamma}\hspace{0.1cm} {\rm for}\hspace{0.1cm} {\rm some}\hspace{0.1cm} {\rm monoform} \hspace{0.1cm}{\rm object}\hspace{0.1cm} H\in\alpha\}.$$
\end{definition}
Clearly $U_{\gamma}=\ASupp(\CT_{\gamma})$ and if $\gamma$ preserves injective objects, it is clear by Lemma \ref{monpres} that $$U_{\gamma}=\{\alpha\in\ASpec\CA|\hspace{0.1cm} H\in\CT_{\gamma}\hspace{0.1cm} {\rm for}\hspace{0.1cm} {\rm every}\hspace{0.1cm} {\rm monoform} \hspace{0.1cm}{\rm object}\hspace{0.1cm} H\in\alpha\}$$$$=\{\alpha\in\ASpec\CA|\hspace{0.1cm} E(\alpha)\in\CT_{\gamma}\hspace{0.1cm}\}.$$ 

\medskip
\begin{lemma}
Let $\gamma$ be a left exact preradical functor on $\CA$. Then $U_{\gamma}$ is an open subset of $\ASpec\CA$.
\end{lemma}
\begin{proof}
Let $\alpha\in U_{\gamma}$. Then there exists a monoform object $H$ of $\CA$ such that $\alpha=\overline{H}$ and $\gamma(H)=H$ (i.e. $H\in\CT_{\gamma}$). We claim that $\ASupp(H)\subset U_{\gamma}$. Given $\beta\in\ASupp(H)$, there exists a monoform object $G$ of $\CA$ and a subobject $K$ of $H$ such that $G$ is a subobject of $H/K$ and $\overline{G}=\beta$. We notice that $H/K\in\CT_{\gamma}$ and since $\gamma$ is left exact, $\CT_{\gamma}$ is closed under subobjects so that $G\in\CT_{\gamma}$. Therefore $\beta\in U_{\gamma}$.
\end{proof}

\medskip
\begin{lemma}\label{ext}
If $\gamma$ is a left exact preradical functor on $\CA$ preserving injective objects, then $\CT_{\gamma}$ is closed under extensions. 
\end{lemma}
\begin{proof}
Given an object $X$ in $\CT_{\gamma}$, since by the assumption $\gamma(E(X))$ is an injective object containing $X$, we have $\gamma(E(X))=E(X)$ so that $E(X)\in\CT_{\gamma}$. Now, assume that $0\To N\To M\To M/N\To 0$ is an exact sequence of objects of $\CA$ such that $N,M/N\in\CT_{\gamma}$. By the first argument, $E(N), E(M/N)\in\CT_{\gamma}$ and using horseshoe lemma, $M$ is a subobject of $E(N)\oplus E(M/N)$. Finally since $\gamma$ is left exact, $\CT_{\gamma}$ is closed under subobjects so that $M\in\CT_{\gamma}$.
\end{proof}

\medskip

\begin{corollary}\label{coc}
If $\gamma$ is a left exact preradical functor preserving injective objects, then $\gamma$ is radical.
\end{corollary}
\begin{proof}
Let $M\in\CA$ and let $\gamma(M/\gamma(M))=N/\gamma(M)$. Then we have the following exact sequence of objects of $\CA$
$$0\To \gamma(M)\To N\To N/\gamma(M)\To 0.$$ Lemma \ref{ext} implies that $N\in\CT_{\gamma}$ and this shows that $\gamma(M)=N$ and so we deduce that $\gamma(M/\gamma(M))=0$.
\end{proof}

\medskip
\begin{proposition}\label{locgam}
 If $\gamma$ is a left exact preradical functor, then $\CT_{\gamma}\subseteq\ASupp^{-1}(U_{\gamma})$. Moreover if $\gamma$ is radical, then $\ASupp^{-1}(U_{\gamma})=\CT_{\gamma}$.
\end{proposition}
\begin{proof}
The first claim follows from $U_{\gamma}=\ASupp(\CT_{\gamma})$ which has been mentioned in Definition \ref{dd}. The second claim is a consequence of [St, Chap. VI, Corollary 1.8] and [K2, Theorem 5.10].
\end{proof}

\medskip

The following theorem shows that any left exact preradical functor can be contained in a section functor. In particular, any radical functor is a section functor.

\begin{theorem}\label{radgam}
Let $\gamma$ be a left exact preradical functor. Then $\gamma$ is a subfunctor of $\Gamma_{U_{\gamma}}$. Moreover, if $\gamma$ is radical, then the equality $\gamma=\Gamma_{U_{\gamma}}$ holds.
\end{theorem}
\begin{proof}
It is enough to show that $\gamma(M)$ is a subobject of $\Gamma_{U_{\gamma}}(M)$. Since $\gamma$ is left exact, it is idempotent; and hence it follows from Proposition \ref{locgam} that $\gamma(M)\in\CT_{\gamma}\subseteq \ASupp^{-1}(U_{\gamma})$ and then so 
 $\Gamma_{U_{\gamma}}(\gamma(M))=\gamma(M)$. But this implies that $\gamma(M)\subseteq\Gamma_{U_{\gamma}}(M)$. To prove the second claim, for any object $M$ of $\CA$, according to Proposition \ref{locgam}, we have $\gamma(\Gamma_{U_{\gamma}}(M))=\Gamma_{U_{\gamma}}(M)$ which implies that $\Gamma_{U_{\gamma}}(M)\subseteq\gamma(M)$. Now this fact along with the first part imply that $\Gamma_{U_{\gamma}}(M)=\gamma(M)$. 
\end{proof}

\medskip
For a preradical functor $\gamma$, we define a subset of $\ASpec\CA$ by the following
$$W_{\gamma}=\{\alpha\in\ASpec\CA|\hspace{0.1cm}{\rm there}\hspace{0.1cm} {\rm exists} \hspace{0.1cm}\beta\in\ASpec\CA\hspace{0.1cm} {\rm with}\hspace{0.1cm} \beta\leq \alpha $$$${\rm  and} \hspace{0.1cm}\gamma(G)\neq 0 \hspace{0.1cm}{\rm for} \hspace{0.1cm}
{\rm some}\hspace{0.1cm}{\rm monoform} \hspace{0.1cm}{\rm object}\hspace{0.1cm}G \hspace{0.1cm}
{\rm of} \hspace{0.1cm}\CA\hspace{0.1cm}{\rm such}\hspace{0.1cm} {\rm that} \hspace{0.1cm}\beta=\overline{G}\}.$$

\begin{proposition}\label{xx}
$W_{\gamma}$ is an open subset of $\ASpec\CA$ if one of the following conditions occur. 
\begin{itemize}
	
	\item [(1)] $\gamma$ is a left exact preradical; in particular in this case $\ASupp(\CT_{\gamma})= W_{\gamma}$.
	\item [(2)] $\ASpec\CA$ is an Alexandroff space.

\end{itemize}
\end{proposition}
\begin{proof}
(1) Assume that $\alpha\in\ASupp(\CT_{\gamma})$. Since $\gamma$ is left exact, there exists a monoform object $H\in\CT_{\gamma}$ 
such that $\alpha=\overline{H}$. Considering $\beta=\alpha$ in the definition, we deduce that $\alpha\in W_{\gamma}$. Conversely, given $\alpha\in W_{\gamma}$, there exists $\beta\in\ASpec\CA$ such that $\beta\leq \alpha$ and $\gamma(G)\neq 0$ for some monoform object $G$ with $\overline{G}=\beta$. Since $\gamma$ is left exact, it is idempotent and so we may assume that $\gamma(G)=G$ so that $G\in\CT_{\gamma}$. The fact that $\beta\leq\alpha$ implies that $\alpha\in\ASupp(G)\subseteq\ASupp(\CT_{\gamma})$. (2) Putting $\Phi=\ASpec\CA\setminus W_{\gamma}$, it is clear that $\Phi=\cup_{\alpha\in \Phi}\overline{\{\alpha\}}$ and since $\ASpec\CA$ is Alexandroff, $\Phi$ is closed. 
\end{proof}

\medskip
\begin{corollary}
If $\gamma$ is a left exact preradical functor, then $U_{\gamma}=W_{\gamma}$.
\end{corollary}
\begin{proof}
 The result is straightforward by using Proposition \ref{xx} and the fact that $U_{\gamma}=\ASupp(\CT_{\gamma})$.  
\end{proof}

\begin{proposition}\label{asssec}
Let $U$ be an open subset of $\ASpec\CA$ such that $\Gamma_U$ preserves injective objects and let $M$ be an object of $\CA$. 
Then $\AAss(M/\Gamma_U(M))\subseteq\AAss(M)$.
\end{proposition}
\begin{proof}
Given $\alpha\in\AAss(M/\Gamma_U(M))$, there exists a monoform object $H$ of $\CA$ and a subobject $N$ of $M$ such that $\alpha=\overline{H}$ and 
$H=N/\Gamma_U(M)$. We observe that $\ASupp(N)\nsubseteq U$ so that $\ASupp(H)\nsubseteq U$. On the other hand, the exact sequence $0\To\Gamma_U(M)\To N\To H\To 0$ implies that $\AAss(N)\subseteq \AAss(\Gamma_U(M))\cup \{\alpha\}$. If $\alpha\notin\AAss(N)$, then $\AAss(N)=\AAss(\Gamma_ U(M))$. By Matlis structure theorem, we have $E(N)=\bigoplus_{\beta\in\AAss(N)}\limits E(\beta)^{(\mu(\beta))}$, where $E(\beta)^{(\mu(\beta))}=\bigoplus_{\mu(\beta)}\limits E(\beta)$ and $\mu(\beta)$ denotes the numbers of $E(\beta)$ appearing in $E(N)$. Therefore, for each $\beta\in\AAss(N)$, Lemma \ref{gammon} implies that $\Gamma_U(E(\beta))=E(\beta)$; and hence $\Gamma_U(E(N))=E(N)$ so that $\Gamma_U(N)=N$ which is a contradiction. 
\end{proof}

\medskip
\begin{lemma}\label{locga}
Let $\alpha$ be an atom in $\ASpec\CA$. Then $t_{\alpha}=\Gamma_{U}$, where $U=\ASupp(\CX(\alpha))$. 
\end{lemma}
\begin{proof}
It is enough to show that for any object $M$, we have $t_{\alpha}(M)=\Gamma_{U}(M)$. Since $t_{\alpha}(M)\in\CX(\alpha)$, we have $\ASupp(t_{\alpha}(M))\subseteq U$ and then $\Gamma_U(t_{\alpha}(M))=t_{\alpha}(M)$ so that $t_{\alpha}(M)\subseteq \Gamma_U(M)$. On the other hand, $\ASupp(\Gamma_U(M))\subseteq U=\ASupp(\CX(\alpha))$. Thus it follows from [K2, Theorem 5.10] that $\Gamma_U(M)\in\CX(\alpha)$ so that $t_{\alpha}(\Gamma_U(M))=\Gamma_U(M)$. The last fact implies that $\Gamma_U(M)\subseteq t_{\alpha}(M)$.
\end{proof}

Let $M$ be an object of $\CA$. An atom $\alpha\in\ASupp(M)$ is called {\it minimal} provided there is no $\beta\in\ASupp(M)$ with $\beta< \alpha$. The set of all minimal atoms of $\ASupp(M)$ is denoted by MinASupp$(M)$. 
\medskip
\begin{corollary} [Sa, Propositions 4.11 and 4.12] \label{minas}
Let $M$ be a noetherian object and let $\alpha\in{\rm MinASupp}(M)$. Then $\alpha\in\AAss(M/{t_{\alpha}}(M))$. In particular if $\CX(\alpha)$ is stable, then $\alpha\in\AAss(M)$.
\end{corollary}
\begin{proof}
Assume that $\alpha\in\MinASupp(M)$. Then by virtue of [K2, Prposition 6.7], the object $M_{\alpha}$ in $\CA_{\alpha}$ has finite length and so $\AAss_{\CA_{\alpha}}(M_{\alpha})=\{\alpha\}$. For any  monoform object $D$ of $\CA$ with $\overline{D}=\alpha$, the object $D_{\alpha}$ is simple in $\CA_{\alpha}$. It is clear that $t_{\alpha}(D)=0$ and hence it follows from [K2, Lemma 5.14] that $G(D_{\alpha})$ is a monoform object of $\CA$ containing $D$ where $G: \CA_{\alpha}\To \CA$ is the right adjoint functor of $(-)_{\alpha}:\CA\To \CA_{\alpha}$. Thus we have $\alpha=\overline{G(D_{\alpha})}$.
 Since $M_\alpha$ has finite length, there exists a composition series 
$$0\subseteq L_1\subseteq\dots\subseteq L_n=M_{\alpha}$$ of subobjects of $M_{\alpha}$ such that for each $1\leq i\leq n$, we have $L_i/L_{i-1}\cong D_{\alpha}$. Applying the left exact functor $G(-)$ to the exact sequences $0\To L_{i-1}\To L_i\To L_i/L_{i-1}\To 0$ and using an easy induction, we deduce that 
$\AAss(G(M_{\alpha}))=\{\alpha\}$. According to [K1, Proposition 4.9], $M/t_{\alpha}(M)$ is essential in $G(M_{\alpha}
)$ and  hence $\AAss(M/t_{\alpha}(M))=\AAss(G(M_{\alpha}))=\{\alpha\}$. The second claim follows from Proposition \ref{stab}, \ref{asssec} and Lemma \ref{locga}.  
 \end{proof}

\medskip
\begin{corollary}\label{ess}
Let $\CA$ be a locally stable category and $M$ be an object of $\CA$. If $N$ is an essential subobject of $M$, then $\ASupp(N)=\ASupp(M)$.
\end{corollary}
\begin{proof}
We first assume that $M$ is noetherian. For any $\alpha\in{\rm MinASupp}(M)$, using Corollary \ref{minas}, $\alpha\in\AAss(M)=\AAss(N)$. Then $\alpha \in{\rm MinASupp}(N)$. Now for any 
$\beta\in\ASupp(M)$, according to [K2, Proposition 4.7], there exists $\alpha\in{\rm MinASupp}(N)$ such that $\alpha\leq \beta$. Then $\beta\in\ASupp(N)$. Now assume that $M$ is an arbitrary object of $\CA$. In this case $M$ is a direct union of its noetherian objects, that is $M=\dot{\bigcup}M_i$. For any $\beta\in\ASupp(M)$, there exists some $i$ such that $\beta\in\ASupp(M_i)$. It is clear that $N\cap M_i$ is essential subobject of $M_i$; hence using the first part $\beta\in\ASupp(N\cap M_i)$. 
\end{proof}
%%%%%%%%%%%%%%%%%%%%%%%%%%%%%%%%%%%%%%%%%%%%%%%%%%%%%%%%%%%%%%%%%%%%%%%%%%%%%%%%%%%%%%%%%%%%%%%%%%%%%%%%%%%%%%%%%%%%%%%%%%%
%%%%%%%%%%%%%%%%%%%%%%%%%%%%%%%%%%%%%%%%%%%%%%%%%%%%%%%%%%%%%%%%%%%%%%%%%%%%%%%%%%%%%%%%%%%%%%%%%%%%%%%%%%%%%%%%%%%%%%%%%%%%%%
%%%%%%%%%%%%%%%%%%%%%%%%%%%%%%%%%%%%%%%%%%%%%%%%%%%%%%%%%%%%%%%%%%%%%%%%%%%%%%%%%%%%%%%%%%%%%%%%%%%%%%%%%%%%%%%%%%%%%%

\medskip

\section{non-vanishing of local cohomology objects}

We start this section by the following definitions.
\begin{definition}
Given an object $M$ of $\CA$, we define the dimension of $M$, denoted by $\dim M$, that is the largest non-negative integer $n$
 such that $\alpha_0<\alpha_1<\dots<\alpha_n$ is a chain of atoms in $\ASupp(M)$.
\end{definition}

An atom $\alpha$ in $\ASpec\CA$ is said to be {\it maximal} if there
exists a simple object $H$ of $\CA$ such that $\alpha=\overline{H}$. 
We denote by ${\rm m-}\ASpec \CA$, the subset of $\ASpec \CA$
consisting of all maximal atoms. 
By virtue of [Sa, Proposition 3.2], an atom $\alpha$ is maximal if and only if $\{\alpha\}$ is an open subset of $\ASpec\CA$. Moreover, if $\frak m$ is an maximal atom, it is maximal under the order relation $\leq$ in $\ASpec\CA$. More precisely, assume that $H$ is a simple object with $\alpha=\overline{H}$ and $\beta\in\ASpec\CA$ such that $\alpha\leq
\beta$. Then, since $\alpha\in\ASupp(H)=\{\alpha\}$, the definition
implies that $\beta\in\ASupp(H)$ and so $\beta=\alpha$.

A definition of the dimension of an object $M$ of $\CA$ was given by Gabriel [Ga, Chapter IV] or [GW, Chapter 15] and nowadays it is called Krull-Gabriel dimension and often denoted by KGdim$M$ which we present as follows.

\medskip
\begin{definition}
For a Grothendieck category $\CA$, we define the {\it Krull-Gabriel filtration} of $\CA$ as follows. For any ordinal number $\alpha$
we denote by $\CA(\alpha)$, the localizing subcategory of $\CA$ is defined in the following manner:

$\CA({-1})$ is the zero subcategory.

$\CA(0)$ is the smallest localizing subcategory containing all simple objects.

Let us assume that $\alpha=\beta+1$ and denote by $T_{\beta}:\CA\To\CA/\CA({\beta})$ the canonical functor and by 
$S_{\beta}:\CA/\CA({\beta})\To\CA$ the right adjoint functor of $T_{\beta}$. Then an object $X$ of $\CA$ will belong to 
$\CA_{\alpha}$ if and only if $T_{\beta}(X)\in{\rm Ob}((\CA/\CA({\beta}))(0))$. If $\alpha$ is a limit ordinal, then $\CA({\alpha})$ is the localizing subcategory generated by all localizing subcategories $\CA_{\beta}$ with $\beta\leq \alpha$. It is clear that if $\alpha\leq \alpha'$, then $\CA({\alpha})\subseteq\CA({\alpha'})$. Moreover, since the class of all localizing subcategories of $\CA$ is a set, there exists an ordinal $\alpha$ such that $\CA({\alpha})=\CA({\tau})$ for all $\tau\geq \alpha$. Let us put $\CA({\tau})=\cup_{\alpha}\CA({\alpha})$.

 We say that the localizing subcategories $\{\CA_{\alpha}\}_{\alpha}$ define the Krull-Gabriel filtration of $\CA$. We say that an object $M$ of $\CA$ has the {\it Krull-Gabriel dimension} defined if $M\in{\rm Ob}(\CA({\tau}))$. The smallest ordinal number $\alpha$ so that $M\in{\rm Ob}(\CA({\alpha}))$ is denoted by ${\rm KGdim}M$.

 It is clear by definition that KGdim$0=-1$ and KGdim$M\leq 0$ if and only if $\ASupp(M)\subseteq {\rm m-}\ASpec \CA$.
\end{definition}

In the following lemma and proposition we obtain the relationship between Krull-Gabriel dimension and our new dimension of objects.

\medskip
\begin{lemma}\label{gone}
Let $M$ be an object of $\CA$. If $\dim M\geq 1$, then ${\rm KGdim}M\geq 1$. 
\end{lemma}
\begin{proof}
If dim$M\geq 1$, there exists a non-maximal atom $\alpha\in\ASupp(M)$. This implies that $\ASupp(M)\nsubseteq{\rm m-}\ASpec \CA$ and so $M\notin{\rm Ob}(\CA_0)$. Thus KGdim$M\geq 1$.
\end{proof}

\medskip
\begin{proposition}
If $M$ is a noetherian object of $\CA$ with $\dim M=1$ and ${\rm MinASupp}(M)$ is a finite set, then ${\rm KGdim}M=1$.
\end{proposition}
\begin{proof}
According to Lemma \ref{gone}, we have ${\rm KGdim}M\geq 1$. Assume that $$M_1\supseteq M_2\supseteq M_3\supseteq\dots$$
is a descending chain of submodules of $M$. For any $\alpha\in{\rm MinSupp}(M)$, if ${M_1}_{\alpha}\neq 0$, then $\alpha\in{\rm MinASupp}(M_1)$ and hence $({M_1})_{\alpha}$ has finite length. Thus there exists a positive integer $n_{\alpha}$ such that for all $i\geq n_{\alpha}$, we have $(M_i/M_{i+1})_{\alpha}=0$. Since ${\rm MinASupp}(M)$ is a finite set, we can get a positive integer $n$ such that 
$(M_i/M_{i+1})_{\alpha}=0$ for all $i\geq n$ and all $\alpha\in {\rm MinASupp}(M)$. Therefore $\ASupp(M_i/M_{i+1})$ contains only maximal atoms and hence it follows from [Sa, Theorem 2.12] that $M_i/M_{i+1}$ have finite length for all $i\geq n$. Then KGdim$(M_i/M_{i+1})=0$ for all $i\geq n$. This implies that KGdim$M=1$. 
\end{proof}

\medskip
 \begin{definition}
Assume that $W$ is an open subset of $\ASpec\CA$ and $\Gamma_W$ preserves injective objects. For any object $M$ of $\CA$ and $i\geq 0$, we define $i$-{\it th local cohomology object of $M$ with respect to} $W$, denoted by $H_W^i(M)$, that is $H_W^i(M)=H^i(\Gamma_W(I))$ where $I$ is an injective resolution of $M$. When $\Gamma_W(M)=M$, since $\Gamma_W$ preserves injective objects, $M$ has an injective resolution with each components $\Gamma_W$-torsion. This implies that $H_W^i(M)=0$ for all $i>0$.  

 In the case $W=\{\frak m\}$ where $\frak m$ is a maximal atom, we denote the $i$-th local cohomology object $M$ with respect to $W$ by $H_{\frak m}^i(M)$. A Grothendieck category is said to be {\it local} 
if $\ASpec\CA$ contains only one maximal atom $\frak m$. In this case the local category $\CA$ is denoted by $(\CA,\frak m)$.

 We notice that if $\CA$ is a local Grothendieck category in sense of [K2, Definition 6.3] or [P, Section 4.20], it is local by our definition. To be more precise, in this case, there exists a simple object $S$ such that $E(S)$ is an injective cogenerator of $\CA$. Then for any other simple object $H$ of $\CA$, $H$ is isomorphic to a subobject of $E(S)$ so that we deduce $H\cong S$. Therefore $\overline{S}$ is the only maximal atom of $\CA$. 

 Conversely, when $\CA$ is a local locally noetherian Grothendieck category,  $\CA$ is local in sense of [K2, P]. Because if $(\CA,\frak m)$ is a local locally noetherian Grothendieck category, then any object $M$ of $\CA$ contains a noetherian subobject $N$. Since $N$ is noetherian, it has a simple quotient object $N/N_1$ and since $(\CA,\frak m)$ is local, we have $\frak m=\overline{N/N_1}$ so that $\frak m\in\ASupp(M)$. It now follows from [K2, Proposition 6.4] that $\CA$ is local in sense of [K2, P].    
\end{definition}

Throughout this section we assume that $(\CA,\frak m)$ is a local category, $\alpha\in\ASpec\CA$ such that $H(\alpha)$ is of dimension one and $M$ is a noetherian monoform object with $\alpha=\overline{M}$. We have the following lemmas.  

\medskip
\begin{lemma}\label{di}
If $H$ is any monoform object such that $\alpha=\overline{H}$. Then $\dim H=1$.
\end{lemma}
\begin{proof}
It follows from Lemma \ref{simp} that $H(\alpha)_{\alpha}$ is simple and hence $\alpha\in{\rm MinASupp}(H(\alpha))$ and since $\dim H(\alpha)=1$, we deduce that $\alpha\neq \frak m$. On the other hand, by the same reasoning $\alpha\in{\rm MinASupp}(H)$ and $\ASupp(H)\subseteq\ASupp(H(\alpha))$. Thus $\dim H=1$. 
\end{proof}

In the rest of this section $\CA$ is a locally stable category. 
\begin{lemma}\label{noeth}
If $H_{\frak m}^1(M)$ is noetherian, then $H_{\frak m}^1(M')$ is noetherian for any noetherian monoform object $M'$ with $\alpha=\overline{M'}$.
\end{lemma}
\begin{proof}
We prove the assertion in the three steps.
Step 1) Assume that $M'$ is a noetherian extension of $M$. It follows from Lemma \ref{di} that $\dim M'=1$. Since $\CA$ is locally stable, it follows from Corollary \ref{minas} that $\ASupp(H)=\{\alpha,\frak m\}$ for any noetherian monoform object $H$ with $\alpha=\overline{H}$. Thus Lemma \ref{simp} implies that $\ASupp(M'/M)=\{\frak m\}$ and hence $\Gamma_{\frak m}(M'/M)=M'/M$. Thus applying the functor $\Gamma_{\frak m}(-)$ to the exact sequence $0\To M\To M'\To M'/M\To 0$ induces an exact sequence $ H_{\frak m}^1(M)\To H_{\frak m}^1(M')\To 0$ which forces that $H_{\frak m}^1(M')$ is noetherian.

Step 2) Assume that $M'$ is a subobject of $M$ and so there is an exact sequence of objects 
$0\To M'\To M\To M/M'\To 0$. By the same reasoning in Step 1, $M/M'$ has finite length and so there is an exact sequence of objects $M/M'\To H_{\frak m}^1(M')\To H_{\frak m}^1(M)\To 0$ which deduces that $H_{\frak m}^1(M')$ is noetherian.

Step 3) Assume that $M'$ is any noetherian monoform object of $\CA$ such that $\alpha=\overline{M'}$. Then $M'$ has a noetherian subobject $L$ which is isomorphic to a subobject of $M$. It follows from Step 2 that $H_{\frak m}^1(L)$ is noetherian and so replacing $M$ by $L$ and using Step 1, the object $H_{\frak m}^1(M')$ is noetherian.   
\end{proof}

\medskip

\begin{corollary}\label{c}
If $H_{\frak m}^1(M)$ is not noetherian, then $H_{\frak m}^1(M')$ is not noetherian for any noetherian monoform object $M'$ with $\alpha=\overline{M'}$. 
\end{corollary}
\begin{proof}
If $H_{\frak m}^1(M')$ is noetherian for some noetherian monoform object $M'$ with $\alpha=\overline{M'}$, then Lemma \ref{noeth} implies that $H_{\frak m}^1(M)$ is noetherian which is a contradiction.  
\end{proof}

\medskip

\begin{lemma}\label{not}
If ${\rm End}(M)$ is not a skew field, then $H_{\frak m}^1(M)$ is not noetherian. 
\end{lemma}
\begin{proof}
 Assume that $H_{\frak m}^1(M)$ is noetherian. There exists a nonzero morphism $f:M\To M$ which is not isomorphism. But since 
$M$ is monoform, $f$ is injective and so $f$ is not surjective. Then there is an exact sequence $0\To M\stackrel{f}\To M\To C\To 0$ and it follows from Lemma \ref{simp} that $C_{\alpha}=0$. Therefore using Corollary \ref{minas}, the object $C$ has finite length. On the other hand, since $\dim M=1$, it follows from Lemma \ref{gammon} that $\Gamma_{\frak m}(M)=0$. Now applying the functor $\Gamma_{\frak m}(-)$ induces the following exact sequence of objects of $\CA$
$$0\To C\To H_{\frak m}^1(M)\stackrel{\underline{f}}\longrightarrow H_{\frak m}^1(M)\To 0.$$  Since $H_{\frak m}^1(M)$ is noetherian, the ascending chain $\ker \underline{f}\subseteq\ker {\underline{f}}^2\subseteq\dots$ of subobjects of $H_{\frak m}^1(M)$ stabilizes and so there exists a positive integer $n$ such that $\ker\underline{f}^n=\ker\underline{f}^{n+1}$. On the other hand, $\underline{f}^n$ is surjective and hence there exists a subobject $X$ of $H_{\frak m}^1(M)$ such that $C=\underline{f}^n(X)$. Therefore $X\subseteq\ker\underline{f}^{n+1}=\ker\underline{f}^n$ and so $C=0$ which is a contradiction. 
\end{proof}

 \medskip

\begin{theorem}\label{nonvanish}
Assume that $N$ is a noetherian object of $\CA$ of dimension one and for any $\alpha\in\AAss(N/\Gamma_{\frak m}(N))$ there exists a monoform object $M$ such that ${\rm End}(M)$ is not a skew field and $\alpha=\overline{M}$. Then $H_{\frak m}^1(N)$ is not noetherian. 
\end{theorem}
\begin{proof}
We notice that $H_{\frak m}^1(N)\cong H_{\frak m}^1(N/\Gamma_{\frak m}(N))$ and so we may assume that $\Gamma_{\frak m}(N)=0$. We also observe that for any $\frak m\neq\alpha\in\ASupp(N)$, we have $\alpha\in{\rm MinASupp}(N)$ and so using Corollary \ref{minas}, $\alpha\in\AAss(N)$. Since $N$ is noetherian, in view of [K3, Theorem 2.9] that there is a filtration of subobjects of $N$
$$0=N_0 \subseteq N_1\subseteq\dots\subseteq N_n=N$$
such that $N_i/N_{i-1}=M_i$ is a monoform object for all $i=1,\dots, n$.  We proceed by induction on $n$ and so we may assume that $n=2$. The exact sequence $0\To M_1\To N_2\To M_2\To 0$ implies that $\Gamma_{\frak m}(M_1)=0$; and hence it follows from Corollary \ref{c} and Lemma \ref{not} that $H_{\frak m}^1(M_1)$ is not noetherian. In view of Lemma \ref{gammon} we have two cases either $\Gamma_{\frak m}(M_2)=0$ or $\Gamma_{\frak m}(M_2)=M_2$. In the first case, applying the functor $\Gamma_{\frak m}(-)$ to the above exact sequence, we conclude that $H_{\frak m}^1(M_1)$ is a non-noetherian subobject of $H_{\frak m}^1(N_2)$ and so $H_{\frak m}^1(N_2)$ is not noetherian. In the second case, $M_2$ has finite length and so there is the following exact sequence of objects of $\CA$ 
$$0\To M_2\To H_{\frak m}^1(M_1)\To H_{\frak m}^1(N_2)\To 0.$$ Now $H_{\frak m}^1(N_2)$ is not noetherian because $H_{\frak m}^1(M_1)$ is not noetherian.
\end{proof}

According to [Ga, p. 428, Proposition 10], if $A$ is a commutative noetherian
ring, then $\CA=$Mod-$A$ is a locally stable category. Moreover, in this case $\CA$ is locally noetherian and $\CA$ is local if and only if $A$ is a local ring. We also observe that the condition on atoms in Theorem \ref{nonvanish} always holds in general. To be more precise, according to [K2, Proposition 2.9], any $\alpha\in\ASpec(A$-Mod$)$ is corresponding to a prime ideal $\frak p\in\Spec A$, that is $\alpha=\overline{A/\frak p}$ and $A/\frak p$ is a monoform module. If $\frak p$ is not maximal ideal, then ${\rm End}_A(A/\frak p)\cong A/\frak p$ is not a field.

%%%%%%%%%%%%%%%%%%%%%%%%%%%%%%%%%%%%%%%%%%%%%%%%%%%%%%%%%%%%%%%%%%%%%%%%%%%%%%%%%%%%%%%%%%%%%%%%%%%%%%%%%%%%%%%%%%%%%%%%%%%%
%%%%%%%%%%%%%%%%%%%%%%%%%%%%%%%%%%%%%%%%%%%%%%%%%%%%%%%%%%%%%%%%%%%%%%%%%%%%%%%%%%%%%%%%%%%%%%%%%%%%%%%%%%%%%%%%%%%%%%%%%%%%
%%%%%%%%%%%%%%%%%%%%%%%%%%%%%%%%%%%%%%%%%%%%%%%%%%%%%%%%%%%%%%%%%%%%%%%%%%%%%%%%%%%%%%%%%%%%%%%%%%%%%%%%%%%%%%%%%%%%%%%%%%5%
\section{Abstract Local cohomology functors}
  
In this section we first recall some definitions and notation from the theory of triangulated categories. For more details we refer readers to [YY]. Let $\CT$ and $\CT'$ be triangulated categories. According to [N, Definition 2.1.1], an additive functor $\delta:\CT\To\CT'$ is called a {\it triangulated functor} provided that $\delta(X[1])\cong\delta(X)[1]$ for any $X\in\CT$ and $\delta$ preserves triangles. For the functor $\delta$, we define full subcategories of $\CT'$ and $\CT$ as  
$${\rm Im}(\delta)=\{X'\in\CT'|\hspace{0.1cm} X'\cong\delta(X) \hspace{0.1cm}{\rm for} \hspace{0.1cm}{\rm some}\hspace{0.1cm} X\in\CT\},$$ 
$${\rm Ker}(\delta)=\{X\in\CT|\hspace{0.1cm}\delta(X)\cong 0\}.$$
	The notion stable $t$-structure is introduced by Miyachi. Recall that a full subcategory of a triangulated category is called a triangulated subcategory if it is closed closed under shift functor [1] and making triangles.  
	\begin{definition}
	A pair $(\CU,\CV)$ of full triangulated subcategories of a triangulated category $\CT$ is called a {\it stable t-structure} on $\CT$ if it satisfies the following conditions.
	
	\begin{enumerate}
		\item $\Hom_{\CT}(U,V)=0$.
		\item For any $X\in\CT$, there is a triangle $U\To X\To V\To U[1]$ with $U\in\CU$ and $V\in\CV$.
	\end{enumerate}
	\end{definition}
	The following theorem due to Miyachi is a key to our argument in this section.

	\begin{theorem}\label{miyachi}{\rm ([M, Proposition 2.6])}
	Let $\CT$ be a triangulated category and $\CU$ be a full triangulated subcategory of $\CT$. Then the following conditions are equivalent for $\CU$. 
		\begin{itemize}

		\item[(1)] There is a full subcategory $\CV$ of $\CT$ such that $(\CU,\CV)$ is a stable t-structure on $\CT$.
		\item[(2)] The natural embedding functor $i:\CU\To \CT$ has a right adjoint $\rho:\CT\To\CU$. 
		\end{itemize}

		If it is the case, setting $\delta=i\circ\rho:\CT\To\CT$, we have equalities 
		
	\begin{center}
	$\CU=\im(\delta)$ and $\CV=\CU^\bot=\Ker(\delta)$.
	\end{center}
	\end{theorem}
	
	We now define abstract local cohomology functor which is a main theme of this section. 
	
	\begin{definition}
	Let $\CT=\mathcal{D}^+(\CA)$ be the derived category of all left bounded complexes of objects of $\CA$ and let $\delta:\CT\To\CT$ be a triangle functor. $\delta$ is said to be an {\it abstract local cohomology functor} if the following conditions are satisfied.
	
	\begin{enumerate}
		\item The natural embedding functor $i:\im(\delta)\To \CT$ has a right adjoint $\rho:\CT\To\im(\delta)$ and $\delta\cong i\circ\rho$ (Hence, by Miyachi's Theorem, $(\im(\delta),\Ker(\delta))$ is a stable t-structure on $\CT$.)
		\item The t-structure $(\im(\delta),\Ker(\delta))$ divides indecomposable injective objects, by which we mean that each indecomposable injective object in $\CA$ belongs to either $\im(\delta)$ or $\Ker(\delta)$.
		\end{enumerate}

	\end{definition}

In the proof of the following proposition, we use some techniques of [YY, Example 1.4].
\medskip
\begin{proposition}\label{radab}
If $\gamma$ is a left exact radical functor preserving injective objects, then ${\bf R}\gamma$ is an abstract local cohomology. 
\end{proposition}
\begin{proof}
 It follows from Theorem \ref{radgam} that there is an open subset $W$ of $\ASpec\CA$ such that $\gamma=\Gamma_W$. For any $\alpha\in W$ and $\beta\in\ASpec\CA\setminus W$, according to Lemma \ref{gammon}, we have $\ASupp(E(\alpha))\subseteq W$ and hence $\Hom_{\CA}(E(\alpha),E(\beta))=0$. On the other hand, Lemma \ref{monpres} implies that $(\im{\bf R}$${\Gamma_W},\Ker{\bf R}\Gamma_W)$ divides indecomposable injective objects. Now the same argument as in [YY, Remark 4] deduces that the functor $\gamma$ is an abstract local cohomology.   
\end{proof}

The following proposition is similar to [YY, Lemma 2.7] where $A$ is a noetherian commutative ring and $\CA=$Mod-$A$. The most important point of view about Mod-$A$ is the fact that Hom preserves localization with respect to prime ideals when the first component is a finitely generated $A$-module, but this argument for an arbitrary locally noetherian Grothendieck category is not reasonable. Because of this, in order to prove the proposition, despite some analogous techniques of [YY, Lemma 2.7], the details of the proof is different in many places. The same fact also holds for Theorem \ref{tttt}.

\medskip

\begin{proposition}
Let $X\in \D^+(\CA)$ and $W$ be an open subset of $\ASpec\CA$ such that $\Gamma_W$ preserves injective objects. Then
\begin{itemize}
	
	\item [(1)]  $X\cong 0$ if and only if ${\bf R}\Hom_{\CA_{\alpha}}(H(\alpha)_{\alpha},X_{\alpha})=0$ for all $\alpha\in\ASpec\CA$.
	
\item [(2)] $X\in\im({\bf R}\Gamma_W)$ if and only if ${\bf R}\Hom_{\CA_{\beta}}(H(\beta)_{\beta},X_{\beta})=0$ for all $\beta\in\ASpec\CA\setminus W$.

\item [(3)] If $\CA$ is locally stable, then $X\in\Ker({\bf R}\Gamma_W)$ if and only if for all $\alpha\in W$ we have 
${\bf R}\Hom_{\CA_{\alpha}}(H(\alpha)_{\alpha},X_{\alpha})=0$.

\end{itemize}
\end{proposition}
\begin{proof}
(1) Assume that $X\ncong 0$. Since $X$ is a left bounded complex,  there exists $i_0\in\Z$ such that $H^{i}(X)=0$ for all $i<i_0$ and $H^{i_0}(X)\neq 0$. Taking $\alpha\in\AAss(H^{i_0}(X))$, there exists a monoform subobject $H$ of $H^{i_0}(X)$ such that $\overline{H}=\alpha$ and so according to Lemma \ref{simp}, $H_{\alpha}=H(\alpha)_{\alpha}$ is a subobject of $H^{i_0}(X_{\alpha})$ so that $\Hom_{\CA_{\alpha}}(H(\alpha)_{\alpha}, H^{i_0}(X_{\alpha}))\neq 0$. Therefore we have the following isomorphism of abelian groups $H^{i_0}({\bf R}\Hom_{\CA_{\alpha}}(H(\alpha)_{\alpha},X_{\alpha}))\cong\Hom_{\D^+(\CA_{\alpha})}(H(\alpha)_{\alpha},X_{\alpha}[i_0])
\cong\Hom_{\CA_{\alpha}}(H(\alpha)_{\alpha},H^{i_0}(X_{\alpha}))$ where the last term is nonzero. Thus ${\bf R}\Hom_{\CA_{\alpha}}(H(\alpha)_{\alpha},X_{\alpha})\neq 0$.

 (2) Given $X\in\im({\bf R}\Gamma_W)$, there exists an injective complex $I\in\D^+(\CA)$ such that 
$X\cong \Gamma_W(I)$. Replacing $I$ by $\Gamma_W(I)$ we may assume that each component of $I$ is direct sum of $E(\alpha)$ in which $\alpha\in W$. For any $\alpha\in W$, using Lemma \ref{gammon}, we have $\Gamma_W(E(\alpha))=E(\alpha)$; and hence $\ASupp(E(\alpha))\subseteq W$. Thus for any $\beta\in\ASpec\CA\setminus W$, we have $E(\alpha)_{\beta}=0$ and so $I_{\beta}=0$. This forces that $X_{\beta}\cong 0$ and so ${\bf R}\Hom_{\CA_{\beta}}(H(\beta)_{\beta},X_{\beta})=0$ for all $\beta\in\ASpec\CA\setminus W$.
Conversely assume that ${\bf R}\Hom_{\CA_{\beta}}(H(\beta)_{\beta },X_{\beta})=0$ for all $\beta\in\ASpec\CA\setminus W$. Since $(\im{\bf R}\Gamma_W,\Ker{\bf R}\Gamma_W)$ is a stable t-structure, there is a triangle 
$${\bf R}\Gamma_W(X)\stackrel{\Phi(X)}\longrightarrow X\To V\To {\bf R}\Gamma_W(X)[1]$$ 
in which $\Phi$ is the unit morphism $i\circ\rho$  to the identity functor in terms of Theorem \ref{miyachi} and $V\in\Ker{\bf R}\Gamma_W$. Considering an injective resolution $I$ of $X$ we have an exact sequence of complexes $$0\To
\Gamma_W(I)\To I\To I/\Gamma_W(I)\To 0.$$  According to [GM, Chap. IV, 13. Lemma], this exact sequence induces a triangle $$\Gamma_W(I)\stackrel{\Phi(I)}\longrightarrow I\To I/\Gamma_W(I)\To\Gamma_W(I)[1]$$ and so $I/\Gamma_W(I)$ is an injective resolution of $V$ whose components are direct sums of $E(\beta)$ with $\beta\in\ASpec\CA\setminus W$. Suppose $V\ncong 0$ and so similar to the proof of (1), there exists $i_0\in\Z$ such that $H^{i_0}(V)\neq 0$ and $H^{i}(V)= 0$ for all $i<i_0$. Taking $\beta\in\AAss(H^{i_0}(V))$ we have ${\bf R}\Hom_{\CA_{\beta}}(H(\beta)_{\beta},V_{\beta})\neq 0$. Since $H^{i_0}(V)=\Ker d^{i_0}/\im d^{i_0-1}$ where $I/\Gamma_W(I)=0\To J^t\stackrel{d^t}\To J^{t+1}\To\dots \To J^{i_0-1}\stackrel{d^{i_0-1}}\To J^{i_0}\stackrel{d^{i_0}}\To\dots$ and $\im d^{i_0-1}$ is an injective object, $H^{i_0}(V)$ is a subobject of $\Ker d^{i_0}\subseteq J^{i_0}$. This implies that $\beta\in\ASpec\CA\setminus W$. Since ${\bf R}\Gamma_W(X)\in\im({\bf R}\Gamma_W)$ it follows from the first part and assumption that ${\bf R}\Hom_{\CA_{\beta}}(H(\beta)_{\beta},X_{\beta})={\bf R}\Hom_{\CA_{\beta}}(H(\beta)_{\beta},{\bf R}\Gamma_W(X)_{\beta})=0$; and hence $
{\bf R}\Hom_{\CA_{\beta}}(H(\beta)_{\beta},V_{\beta})=0$ which is a contradiction. Thus $V\cong 0$ and so $X\cong{\bf R}\Gamma_W(X)\in\im{\bf R}\Gamma_W$. 

 (3) Suppose that $X\in\Ker{\bf R}\Gamma_W$ and so ${\bf R}\Gamma_W(X)\cong 0$. Taking an injective resolution $I$ of $X$, we deduce that $\Gamma_W(I)$ is exact and so $X$ is quasi-isomorphic to $I/\Gamma_W(I)$. Replacing $I$ by $I/\Gamma_W(I)$ we may assume that $I$ consists of injective objects $E(\beta)$ such that $\beta\in\ASpec\CA\setminus W$. 
 Therefore it suffices to show that for any $\alpha\in W$ and $\beta\in\ASpec\CA\setminus W$, we have $$\Hom_{\CA_{\alpha}}(H(\alpha)_{\alpha},E(\beta)_{\alpha})=0.$$  If $E(\beta)_{\alpha}=0$, there is nothing to prove and so we may assume that $E(\beta)_{\alpha}\neq 0$. Since $\CX(\alpha)$ is stable MinASupp$(E(\beta))=\{\beta\}$ and since $\alpha\in\ASupp(E(\beta))$, we have $\beta\leq \alpha$.
 Suppose that $\Hom_{\CA_{\beta}}(H(\alpha)_{\alpha},E(\beta)_{\alpha})\neq 0$. By the adjointness we have the following isomorphism of the abelian groups $$\Hom_{\CA_{\alpha}}(H(\alpha)_{\alpha},E(\beta)_{\alpha})\cong \Hom_{\CA}(H(\alpha),GF(E(\beta)))$$
 where $F:\CA\To \CA_{\alpha}$ is the left adjoint functor with its right adjoint functor $G:\CA_{\alpha}\To\CA$.
We notice that $t_{\alpha}(E(\beta))=0$; otherwise we have $\ASupp(t_{\alpha}(E(\beta)))\subseteq\ASupp(\CX(\alpha))$ which is a contradiction as $\beta\notin\ASupp(\CX(\alpha))$. Thus $E(\beta)$ is an essential subobject of $GF(E(\beta))$. Now, if $f$ is any nonzero morphism in $\Hom_{\CA}(H(\alpha),GF(E(\beta)))$, then $\im f\cap E(\beta)\neq 0$ so that $\beta\in\AAss(\im f)$; but $\ASupp(\im f)\subseteq\ASupp(H(\alpha))\subseteq\ASupp(E(\alpha))\subseteq W$ which is a contradiction. Therefore ${\bf R}\Hom_{\CA_{\alpha}}(H(\alpha)_{\alpha},X_{\alpha})\cong\Hom_{\CA_{\alpha}}(H(\alpha)_{\alpha},I_{\alpha})=0$ for all $\alpha\in W$. Conversely assume that ${\bf R}\Hom_{\CA_{\alpha}}(H(\alpha)_{\alpha},X_{\alpha})=0$ for all $\alpha\in W$ and take a triangle $${\bf R}\Gamma_W(X)\To X\To V\To {\bf R}\Gamma_W(X)[1]$$ such that $V\in\Ker{\bf R}\Gamma_W$. Hence ${\bf R}\Gamma_W(X)_{\alpha}\To X_{\alpha}\To V_{\alpha}\To {\bf R}\Gamma_W(X)_{\alpha}[1]$ is a triangle in $\D^+(\CA_{\alpha})$ for all $\alpha\in\ASpec\CA$. It follows from the first part that ${\bf R}\Hom_{\CA_{\alpha}}(H(\alpha)_{\alpha},V_{\alpha})=0$ for all $\alpha\in W$. Hence for any $\alpha\in W$, it follows from the previous triangle and the hypothesis that $${\bf R}\Hom_{\CA_{\alpha}}(H(\alpha)_{\alpha},{\bf R}\Gamma_W(X)_{\alpha})=0.$$ Moreover, by (2) we have ${\bf R}\Hom_{\CA_{\alpha}}(H(\alpha)_{\alpha},{\bf R}\Gamma_W(X)_{\alpha})=0$ for all $\alpha\in\ASpec\CA\setminus W$. Hence (1) implies that ${\bf R}\Gamma_W(X)\cong 0$.
\end{proof}

\medskip
\begin{lemma}\label{ker}
Let $W$ be an open subset of $\ASpec\CA$ such that $\Gamma_W$ preserves injective objects and let $X\in\Ker{\bf R}\Gamma_W$ such that ${\bf R}\Hom_{\CA}(X,E(\alpha))=0$ for all $\alpha\in\ASpec\CA\setminus W$. Then $X\cong 0$.
\end{lemma}
\begin{proof}
The proof is similar to the proof of [YY, Lemma 2.9(1)]. 
\end{proof}

\medskip

\begin{theorem}\label{tttt}
Let $\CA$ be locally stable and let $\delta$ be an abstract local cohomology on $\D^+(\CA)$. Then there exists an open subset $W$ of $\ASpec\CA$ such that $\delta={\bf R}\Gamma_W$.
\end{theorem}
\begin{proof}
Consider $W=\{\alpha\in\ASpec\CA| \hspace{0.1cm}E(\alpha)\in\im\delta\}$. We notice that $(\im\delta,\Ker \delta)$ is a stable t-structure. Assume that $\alpha\in W$ and $\beta\in\ASpec\CA$ such that $\alpha\leq \beta$. We want to show that 
$E(\beta)\in\im\delta$ and so $\beta \in W$. 
Assume that $E(\beta)\notin\im\delta$. As $(\im\delta,\Ker \delta)$ divides the indecomposable injective objects, $E(\beta)\in\Ker\delta$ so that $\Hom_{\D^+(\CA)}(E(\alpha),E(\beta))= 0$. On the other hand, since $\alpha\leq \beta$, there exist monoform objects $G$ and $H$ of $\CA$ such that $\alpha=\overline{H}$ and $\beta=\overline{G}$ and $G=H/K$ for some subobject $K$ of $H$. Therefore the nonzero morphism $H\twoheadrightarrow G\hookrightarrow E(\beta)$ induces a nonzero morphism $E(\alpha)\To E(\beta)$ which is a contradiction. We now assert that $W$ is an open subset of $\ASpec\CA$. Given $\alpha\in W$ and monoform object $H$ of $\CA$ with $\alpha=\overline{H}$, we show that $\ASupp(H)\subseteq W$. Assume that $\beta\in\ASupp(H)$ but $\beta\notin W$. By the first argument $\alpha\nleq \beta$; and hence $\alpha\in\ASupp(\CX(\beta))$ so that there exists a monoform object $H_1\in\CX(\beta)$ such that $\alpha=\overline{H_1}$. Since $\CA$ is locally stable, $E(\alpha)\in\CX(\beta)$ and so $E(\alpha)_{\beta}=0$. Therefore, using [K2, Proposition 6.2], $\beta\notin\ASupp(E(\alpha))$ which contradicts the fact that $\beta\in\ASupp(H)$. As $\CT_{\Gamma_W}=\ASupp^{-1}(W)$ is a localizing subcategory, in view of the assumption it is stable and so using Proposition \ref{stab}, the functor $\Gamma_W$ preserves injective objects. 
For every $\alpha\in W$, we have $E(\alpha)\in\im\delta\cap\im{\bf R}\Gamma_W$ and for every $\beta\in\ASpec\CA\setminus W$, we have $E(\beta)\in\Ker\delta\cap\Ker{\bf R}\Gamma_W$. By Theorem \ref{miyachi} , it is enough to show that $\im\delta=\im{\bf R}\Gamma_W$. We first show that $\im\delta\subseteq{\bf R}\Gamma_W$. To do this, assume that $X\in\im\delta$. Then there is a triangle $${\bf R}\Gamma_W(X)\To X\To V\To{\bf R}\Gamma_W(X)[1]$$ where $V\in\Ker{\bf R}\Gamma_W$. Given $\beta\in\ASpec\CA\setminus W$, we have $E(\beta)\in\Ker\delta\cap\Ker{\bf R}\Gamma_W$ and so $\Hom_{\D^+(\CA)}(X,E(\beta)[n])=0=\Hom_{D^+(\CA)}({\bf R}\Gamma_W(X),E(\beta)[n])$ for any integer $n$. Then viewing the above triangle we have $\Hom_{\D^+(\CA)}(V,E(\beta)[n])=0$ for any integer $n$ and so $H^n({\bf R}\Hom_{\CA}(V,E(\beta)))\cong\Hom_{\D^+(\CA)}(V,E(\beta)[n])=0$ for any integer $n$. Since $V\in\Ker{\bf R}\Gamma_W$, by using Lemma \ref{ker}, we have $V\cong 0$ so that $X\cong {\bf R}\Gamma_W(X)$. Now we prove that $\im{\bf R}\Gamma_W\subseteq\im\delta$. Given $X\in\im{\bf R}\Gamma_W$, there exists a  triangle $$\delta(X)\To X\To Y\To\delta(X)[1]\hspace{0.2cm}(\dag)$$ with $Y\in\Ker\delta$. By the first inclusion we have $\delta(X)\in\im{\bf R}\Gamma_W$ and so ${\bf R}\Gamma_W(\delta(X))=\delta(X)$ and hence there is an traingle $\delta(X)\To X\To{\bf R}\Gamma_W(Y)\To \delta(X)[1]$. The natural morphism ${\bf R}\Gamma_W\To {\bf 1}$ of functors on $\D^+(\CA)$ implies that $Y\cong{\bf R}\Gamma_W(Y)$. Hence $Y$ has an injective resoultion $I$ such that $\alpha\in W$ for any injective indecomposable $E(\alpha)$ appearing in any component of $I$.  Therefore we deduce that $0=\delta(Y)\cong\delta(I)\cong I\cong Y$ and so viewing $(\dag)$,  we have $X\cong\delta (X)$. 
\end{proof}

%%%%%%%%%%%%%%%%%%%%%%%%%%%%%%%%%%%%%%%%%%%%%%%%%%%%%%%%%%%%%%%%%%%%%%%%%%%%%%%%%%%%%%%%%%%%%%%%%%%%%%%%%%%%%%%%%%%%%%%%%%%%%%%
%%%%%%%%%%%%%%%%%%%%%%%%%%%%%%%%%%%%%%%%%%%%%%%%%%%%%%%%%%%%%%%%%%%%%%%%%%%%%%%%%%%%%%%%%%%%%%%%%%%%%%%%%%%%%%%%%%%%%%%%%%%%%%%%%%%%

\end{document}